\documentclass[12pt,reqno]{amsart} 

\usepackage{amsmath}
\usepackage{amssymb,color,enumerate,accents,mathtools}
\usepackage{amsthm}
\usepackage[margin=1.4in]{geometry}
\usepackage{hyperref}
\hypersetup{
colorlinks=true,
citecolor=cyan,
linkcolor=magenta
}
\usepackage{calligra}
\usepackage{xifthen}

\newtheorem*{corollary*}{Corollary}
\newtheorem*{theorem*}{Theorem}

\newtheorem{theorem}{Theorem}
\newtheorem{proposition}{Proposition}
\newtheorem{corollary}{Corollary}
\newtheorem{lemma}{Lemma}

\theoremstyle{definition} 

\newtheorem*{remark*}{Remark}

\newcommand{\pnum}[1]{\textup{(}#1\textup{)}}

\newcommand{\NN}{\mathbb N}
\newcommand{\ZZ}{\mathbb Z}
\newcommand{\QQ}{\mathbb Q}
\newcommand{\RR}{\mathbb R}
\newcommand{\CC}{\mathbb C}

\newcommand{\DD}{\mathbb D}
\renewcommand{\a}{\mathfrak a}

\newcommand{\id}{\operatorname{id}} 

\newcommand{\primes}[1]{M_{#1}} 
\newcommand{\primesarch}[1]{M_{#1}^\infty} 
\newcommand{\nfi}{K} 
\newcommand{\dgr}[1]{d} 
\newcommand{\localdgr}[1]{d_{#1}} 

\newcommand{\varep}{\varepsilon}

\newcommand{\ds}{\displaystyle}
\newcommand{\tik}{\[\begin{tikzcd}}
\newcommand{\zcd}{\end{tikzcd}\]}

\newcommand{\defn}[1]{\emph{#1}}

\newcommand{\INT}[1][]{
    \ifthenelse{\isempty{#1}{}}
    {{\mathcal{I}} } 
    {{\mathcal{I}} }  
}
\newcommand{\ul}{\underline}
\newcommand{\Bdd}{\ell^\infty}

\newcommand{\seq}[1][]{
    \ifthenelse{\isempty{#1}{}}{{s}}{s({#1})}%
}
\newcommand{\seqtwo}[1][]{
    \ifthenelse{\isempty{#1}{}}{{t}}{t({#1})}%
}
\newcommand{\findiff}[1][]{
    \ifthenelse{\isempty{#1}{}}{{c}}{c({#1})}%
}
\newcommand{\dd}[2]{\delta_{#1}{#2}} 
\newcommand{\norm}[1][]{
    \ifthenelse{\isempty{#1}{}}
        {\left\lVert \,\cdot\, \right\rVert}%
        {\left\lVert#1\right\rVert}%
}
\newcommand{\valn}[1][]{
    \ifthenelse{\isempty{#1}{}}
    {\nu}
    {\nu_{#1}} 
    }

\newcommand{\embedding}[1][]{
    \ifthenelse{\isempty{#1}{}}
    {\sigma}
    {\sigma_{#1}} 
    }

\newcommand{\Mod}[1]{\ (\mathrm{mod}\ #1)}
\newcommand{\chebyshev}{\vartheta} 
\newcommand{\harm}[1]{H_{#1}} 
\newcommand{\place}{v} 
\newcommand{\Ndiag}[1]{\Delta_{#1}} 

\mathtoolsset{showonlyrefs} 

\begin{document}

\newcommand{\note}[1]{{\color{magenta}[[#1]]}}

\newcommand{\grantinfo}{National Science Foundation RTG grant DMS-1840234}
\newcommand{\papertitle}{\large{Functions with integer-valued divided differences}}
\newcommand{\shortpapertitle}{{FUNCTIONS WITH INTEGER-VALUED DIVIDED DIFFERENCES}}

\title[\shortpapertitle]{\papertitle}

\author{Andrew O'Desky}
\address{
\parbox{0.5\linewidth}{
    Department of Mathematics\\
    Princeton University\\
    Princeton, NJ 08544-1000, USA\\[.1em]
    \href{http://www.andrewodesky.com/}{\texttt{http://www.andrewodesky.com/}}\\[-.5em]}
}

\date{\today}

\thanks
{This work was supported in part by {\grantinfo}.}

\begin{abstract} 
Let 
$s_0,s_1,s_2,\ldots$ be 
    a sequence of rational numbers whose 
    $m$th divided difference is integer-valued. 
We prove that $s_n$ is a polynomial function in $n$ 
    if $s_n \ll \theta^n$ 
    for some positive number $\theta$ satisfying 
    $\theta < 
    e^{1 + \tfrac{1}{2} + \cdots+ \tfrac{1}{m}} -1$.  
\end{abstract} 


\maketitle

\section{Introduction} 
\label{sec:intro}

In 1971, 
{Hall} and {Ruzsa}
independently discovered
an elegant characterization of polynomial functions among 
congruence-preserving functions. 

\begin{theorem*}[Hall--Ruzsa]
Let $\seq \colon \NN \to \QQ$ be a function. 
Suppose that\footnote
    {We write $r_n \ll q_n$ to mean there is a constant $C>0$ 
    such that $|r_n| \leq C |q_n|$ for all $n$.}
\begin{enumerate}
    \item[\pnum{i}] $\seq$ is integer-valued,  
    \item[\pnum{ii}] $\seq[n+k] \equiv \seq[n] \Mod k$ 
            for all $n,k \in \NN$, and  
        \item[\pnum{iii}] $\seq[n] \ll \theta^n$ 
            for some positive number $\theta$ satisfying
            $\theta<e-1$.  
\end{enumerate}
Then $\seq[n]$ is a polynomial in $n$. 
\end{theorem*}


In this article we generalize this theorem 
    by using the integrality of higher divided differences. 
The first two hypotheses of the Hall--Ruzsa theorem are equivalent to requiring that 
$\dd{0}{\seq}:=\seq$ and $\dd{1}{\seq}$ are integer-valued, 
where 
$$
\dd{1}{\seq}(n,m):=
\frac{\seq[n]-\seq[m]}{n-m}
\quad
(n \neq m)
$$
is the \defn{first divided difference of $\seq$}. 

\vspace{.5em}
Our main result is the following theorem. 


\begin{theorem}\label{thm:introthm}
    Let $s \colon \NN \to \QQ$ be a function. 
    Suppose that 
    \begin{enumerate}
        \item[\pnum{i}] the $m$th divided difference $\dd{m}{\seq}$ of $\seq$ is integer-valued, and 
        \item[\pnum{ii}] $\seq[n] \ll \theta^n$ 
            for some positive number $\theta$ satisfying 
        \begin{equation} 
            \theta < e^{1 + \tfrac{1}{2} + \cdots+ \tfrac{1}{m}} -1. 
        \end{equation} 
    \end{enumerate}
    Then $\seq[n]$ is a polynomial in $n$. 
\end{theorem}


Recall the \defn{$m$th divided difference of $\seq$} is 
the function $\dd{m}{\seq}$ defined by 
\begin{equation*}
\dd{m}{\seq} (n_0,\ldots,n_m) := 
\sum_{i=0}^m
\Big\{
    \prod_{j\neq i}
        (n_i-n_j)^{-1}
\Big\}
\seq[n_i]
\quad
    (n_i \text{ distinct}).
\end{equation*}
Divided differences are a classical construction 
of basic importance in interpolation theory and 
nonarchimedean analysis  
(cf. e.g. 
\cite{milne-1951}, 
\cite{schikhof}, 
\cite{bhargava-2009}). 



\subsection{Outline of the paper} 

The first part of this paper (\S\ref{sec:divdiff}) consists of 
a local analysis of $\dd{m}{\seq}$ at nonarchimedean places. 
Our first new result is a formula for the $p$-adic supremum 
$\norm[\dd{m}{\seq}]_p$ of $\dd{m}{\seq}$. 
To state our formula, we make use of a (possibly new) 
elementary arithmetic function. 
Let $\tau_{p,m}(n)$ denote 
the largest possible $p$-adic 
valuation of a product of $m$ distinct positive 
integers which are all less than or equal to $n$. 
(Proposition~\ref{prop:tauexplicitformula} gives 
an explicit formula for $\tau_{p,m}$.)
Let $\seq \colon \NN \to \QQ_p$ be a function 
and set 
$
\findiff[n] := 
\sum_{k = 0}^n 
\binom{n}{k}(-1)^{n-k}\seq[k]$. 
We show that 
\begin{equation}\label{eqn:supphim}
    \norm[\dd{m}{\seq}]_p
=
\sup_{n \geq m}
|\findiff[n]|_pp^{\tau_{p,m}(n)}. 
\end{equation}
This formula gives a Mahler-type criterion 
for the $p$-integrality of $\dd{m}{\seq}$. 
Along the way, we prove a generalization of Mahler's theorem to bounded, 
not-necessarily-continuous functions (Theorem~\ref{thm:boundedmahler}). 

The second part of this paper 
(\S\ref{sec:asymptotics}) consists of 
a global calculation involving some analytic estimates 
of certain sums over primes. 
Our third new result is that 
\begin{equation} 
    \label{eqn:tauasymptoticvalue}
\sum_{p \leq n}
    {\tau_{p,m}(n)}\log p 
= 
{n \harm{m} + O(n \exp\{-\alpha (\log n)^{1/2} \} \log n )} 
\end{equation} 
    for some positive constant $\alpha$. 

In the third part of this paper (\S\ref{sec:proofs}) 
we combine the local and global results ---  
\eqref{eqn:supphim} and 
\eqref{eqn:tauasymptoticvalue} --- 
to prove Theorem~\ref{thm:introthm}. 
(In fact we prove 
a more general version 
which also includes the main result of \cite{hill-stra-1970}.) 

The $p$-integrality of divided differences is 
a natural condition on functions $s \colon \NN \to \QQ_p$ 
which appears to have not been given a thorough treatment in 
the present literature on interpolation and $p$-adic analysis. 
Therefore it seems worthwhile to us 
to explore this condition somewhat. 
In the last section (\S\ref{sec:interpretations}), 
which is independent from the rest of the paper, 
we give two interpretations of this condition. 



\subsection*{Notation} 
A place $\place$ of a number field $\nfi$ is an equivalence 
class of isometric embeddings 
$\embedding \colon \nfi \to \CC_{p}$ with 
$p \in \primes{\QQ}=\{ 2,3,5, \ldots, \infty\}$ where 
$\CC_p$ is the metric completion of an algebraic closure of 
the $p$-adic field $\QQ_p$ (where $\QQ_\infty := \RR$). 
The set of all places of $\nfi$ is denoted by $\primes{\nfi}$. 
The local degree at $\place$ is denoted by $\localdgr{\place}$. 
We write $|\cdot|_p$ for 
the unique norm on $\CC_p$ extending the norm on $\QQ_p$, 
normalized by $|p|_p = 1/p$. 
The norm associated to a place $\place$ of $K$ 
extending $p \in \primes{\QQ}$ 
is defined by $|x|_\place:=|\embedding(x)|_{p}$ 
for a representative embedding $\embedding$ of $\place$. 
For nonarchimedean places, 
the additive $\place$-adic valuation 
is defined by $\valn[\place]:=\log_{p} |\cdot|_\place^{-1}$ 
where $\log_{p}$ is the base-$p$ logarithm. 
We write $[\,\cdot\,]$ for the floor function   
and $\harm{m} = 1 + \tfrac12 + \cdots + \tfrac1m$ 
for the $m$th harmonic number. 
We write $\chebyshev(n)= \sum_{p \leq n} \log p$ for the 
Chebyshev function and $\pi(n)$ for 
the number of rational primes less than or equal to $n$.  
We define $\Ndiag{m} := 
\{(n_0,\ldots,n_m) \in \NN^{m+1} : 
n_i \text{ all distinct} \}.$  


\section{Background from difference calculus and \texorpdfstring{$p$}{p}-adic analysis}\label{sec:background}  

Let $m$ be a non-negative integer. 
Define $$
\Ndiag{m} := 
\{(n_0,\ldots,n_m) \in \NN^{m+1} : 
n_i \text{ all distinct} \}. 
$$
Let $s \colon \NN \to E$ be a function valued in a $\QQ$-algebra $E$. 
Recall the \defn{$m$th divided difference of $\seq$} 
is the function $\dd{m}{\seq} \colon \Ndiag{m} \to E$ 
given by 
\begin{equation}\label{eqn:divdiff}
\dd{m}{\seq} (n_0,\ldots,n_m) := 
\sum_{i=0}^m
\Big\{
    \prod_{j\neq i}
        (n_i-n_j)^{-1}
\Big\}
\seq[n_i].
\end{equation}
(It is well-known that $\seq$ 
can be reconstructed using the values of its 
divided differences by means of 
Newton's interpolation formula 
(cf.~\eqref{eqn:NIF} or \cite{milne-1951}, \S1).) 
Recall the \defn{$n$th finite difference of $\seq$} is 
defined by\footnote 
{Strictly speaking, this is the sequence obtained by 
evaluating the finite differences of $\seq$ at zero. 
} 
\begin{equation}\label{eqn:defnfindiff}
\findiff[n] := 
\sum_{k = 0}^n 
\binom{n}{k}(-1)^{n-k}\seq[k].
\end{equation}

\begin{lemma}\label{lemma:classical}
Let $\seq \colon \NN \to E$ be a function 
and let $\findiff \colon \NN \to E$ be 
    its finite differences.
Then $\seq$ is polynomial 
if and only if  
$\findiff$ is eventually zero. 
\end{lemma}

\begin{proof} 
Let $S$ be the forward shift operator on sequences defined by 
$(S\seq)(n):= \seq(n+1)$ for all non-negative integers $n$. 
Then for any non-negative integer $\ell$, 
$$
\{(S-\id)^n\seq\}(\ell)
=
\sum_{k = 0}^n 
\binom{n}{k}(-1)^{n-k}\seq[\ell+k], 
$$
and in particular, $\{(S-\id)^n\seq\}(0) = \findiff[n]$. 
We have that $(S-\id)(n^d) = dn^{d-1} + O(n^{d-2})$, 
and so the restriction of $S-\id$ to 
the space of polynomial sequences is nilpotent. 
This shows that $\findiff$ is eventually zero 
if $\seq$ is polynomial.  

Conversely, assume that $\findiff$ is eventually zero. 
It is easy to verify that the inverse relation of 
\eqref{eqn:defnfindiff} is given by 
\begin{equation}\label{eqn:inverserelation}
\seq[n]=
\sum_{k = 0}^n 
    \binom{n}{k}
    \findiff[k],
\end{equation}
and this shows that $\seq$ is given by a polynomial of degree 
$\leq N$ if $\findiff[n] = 0$ for $n > N$. 
\end{proof} 

We recall a classical theorem of Mahler. 
Let $\binom{x}{n} := \frac{x(x-1)\cdots(x-n+1)}{n!}\in\QQ[x]$. 

\begin{theorem*}[Mahler~\cite{mahler}]
Let $\seq \colon \NN \to \CC_p$ be a function 
and let $\findiff \colon \NN \to \CC_p$ denote 
    its finite differences.
Then $\seq$ extends to a continuous function 
$f \colon \ZZ_p \to \CC_p$ if and only if 
    $\ds\lim_{n \to \infty} |\findiff[n]|_p = 0$. 
Furthermore, 
    if $\seq$ extends to a continuous function 
    $f \colon \ZZ_p \to \CC_p$, 
    then the `Mahler series' 
\begin{equation}
    \sum_{n = 0}^\infty
    \findiff[n]
        \binom{x}{n} 
\end{equation}
    converges uniformly to $f$ in $C(\ZZ_p,\CC_p)$ 
and 
\begin{equation}
    \sup_{x \in \ZZ_p} |f(x)|_p
    =
    \sup_{n \geq 0}
    |\findiff[n]|_p.
\end{equation}
\end{theorem*}


\section{The \texorpdfstring{$p$}{p}-adic size of divided differences} 
\label{sec:divdiff} 

Let $p$ be a fixed prime integer. 
Let $E$ be a finite-degree field extension of $\QQ_p$ and 
let $|\cdot|$ denote the norm on $E$ 
which extends the usual $p$-adic norm on $\QQ_p$. 
Let $\seq \colon \NN \to E$ be a function 
and set 
$\findiff[n] = 
\sum_{k = 0}^n 
\binom{n}{k}(-1)^{n-k}\seq[k]$. 
The goal of this section is  
to prove the formula 
\begin{equation}\label{eqn:lemmaeqnnorm2}
\sup_{0 \leq n_0<\cdots<n_m}
    |\dd{m}{\seq}(n_0,\ldots,n_m)| = 
\sup_{n \geq m}
|\findiff[n]|p^{\tau_{p,m}(n)} 
\end{equation}
where $\tau_{p,m}(n)$ is the maximal $p$-adic 
valuation of a product of $m$ distinct positive 
integers $\leq n$, i.e.
$$\tau_{p,m}(n):= 
\max_{\substack{S \subset \{1,\ldots,n\},\\ \# S = m}}
\valn[p]\Big(\prod_{k \in S} k\Big ).
$$

\subsection{The Mahler series of a bounded function} 

Let $\Bdd(\NN^{m+1})$ denote the $E$-Banach space of 
bounded functions $F \colon \NN^{m+1} \to E$ 
under pointwise addition and scalar multiplication 
equipped with the norm 
$$
\norm[F]:= 
\sup_{\ul{n}\in \NN^{m+1}}
|F(\ul{n})|. 
$$
This $p$-adic Banach space contains 
two important closed subspaces: 
\begin{equation} 
    \mathrm{C}(\NN^{m+1}):= \{F \colon \NN^{m+1} \to E \mid 
        \text{$F$ extends to a continuous function 
        $\ZZ_p^{m+1} \to E$}\}
\end{equation} 
and 
\begin{equation} 
    \mathrm{c}_0(\NN^{m+1}):=
    \{C \colon \NN^{m+1} \to E \mid 
        \text{$C$ vanishes at infinity\footnotemark}\}.
\end{equation} 
\footnotetext{i.e., $\ds\lim_{N \to \infty} 
        \sup_{j_0 + \cdots + j_m > N} \big|C({\ul j})\big| = 0 $.}
Let $\ul x = (x_0,\ldots,x_m)$ be indeterminates, 
$\ul j = (j_0,\ldots,j_m)$ be non-negative integers, and 
define $\binom{\ul{x}}{\ul{j}} :=
\binom{x_0}{j_0}
\binom{x_1}{j_1}
\cdots
\binom{x_m}{j_m} \in \QQ[x_0,\ldots,x_m].$ 
Mahler's theorem easily generalizes to multivariate functions 
(cf. \cite{amice-1964}, Corollaire~1, \S2.7) 
and amounts to the assertion that there is a unique isometry 
\begin{align}\label{eqn:isometrymahler} 
\mathcal{M}_0\colon \mathrm{C}(\NN^{m+1}) &\xrightarrow{\sim} \mathrm{c}_0(\NN^{m+1}) \\
    F &\mapsto C\nonumber
\end{align} 
with the property that for all $F\in\mathrm{C}(\NN^{m+1})$ 
the infinite series 
\begin{equation} 
    \sum_{\ul{j} \in \NN^{m+1}} 
        C(\ul{j})\binom{\ul{x}}{\ul{j}}
\end{equation} 
converges in $\mathrm{C}(\NN^{m+1})$ to $F$.

We are interested in computing the norm of $\dd{m}{\seq}$ 
(strictly speaking, $\dd{m}{\seq}$ is not defined on 
all of $\NN^{m+1}$, 
but this technical issue will be easily resolved later). 
Mahler's theorem is therefore insufficient for our purposes, 
since $\dd{m}{\seq}$ is only assumed to be $p$-integral 
and inhabits neither $\mathrm{C}(\NN^{m+1})$ nor $\mathrm{c}_0(\NN^{m+1})$ 
in general. 

To resolve this problem, we will show that the isometry 
$\mathcal M_0$
extends to a self-isometry of $\Bdd(\NN^{m+1})$. 
The following notion will be helpful: 
if $F,C \colon \NN^{m+1} \to E$ are functions, 
not necessarily bounded or continuous, 
we say that 
\begin{equation} 
    \sum_{\ul{j} \in \NN^{m+1}} 
        C(\ul{j})\binom{\ul{x}}{\ul{j}} 
\end{equation} 
is a \defn{generalized Mahler series for $F$} if 
\begin{equation}\label{eqn:finitesum} 
    F(\ul{n}) = \sum_{\ul{j} \in \NN^{m+1}} 
        C(\ul{j})\binom{\ul{n}}{\ul{j}}
        \quad \text{(finite sum)}
\end{equation} 
for all $\ul{n} \in \NN^{m+1}$. 
The values of $C$ are called the 
\defn{Mahler coefficients of $F$}. 
The following theorem generalizes Mahler's theorem to bounded functions 
which are not necessarily continuous. 

\begin{theorem}\label{thm:boundedmahler}
Every function $F \colon \NN^{m+1} \to E$ 
    admits a generalized Mahler series with 
    unique Mahler coefficients. 
Furthermore, 
    the mapping $F \mapsto C$ 
    which sends a bounded function $F$ to 
    its Mahler coefficients $C$ 
    extends $\mathcal{M}_0\colon \mathrm{C}(\NN^{m+1}) \xrightarrow{\sim} \mathrm{c}_0(\NN^{m+1})$ 
to a self-isometry $\mathcal M$ of $\Bdd(\NN^{m+1})$. 
\end{theorem}


Note however that the proposition does not imply that 
$\{\binom{\ul{x}}{\ul{j}}\}_{\ul{j}\in\NN^{m+1}}$ 
is an orthonormal Banach basis for $\Bdd(\NN^{m+1})$,  
as generalized Mahler series are typically divergent 
in $\Bdd(\NN^{m+1})$. 

\begin{proof} 

We proceed by induction on $m$. 
    If $m = 0$ then we take $C(n)$ to be the 
$n$th finite difference of $F$ given by 
$C(n) := 
\sum_{k = 0}^n 
\binom{n}{k}(-1)^{n-k}F(k).$ 
The inverse relation is given by \eqref{eqn:inverserelation} 
and uniqueness follows. 
If $F$ is bounded then 
$C$ is bounded by the ultrametric inequality and vice versa.   
This mapping restricts to an isometry 
$\mathrm{C}(\NN) \xrightarrow{\sim} \mathrm{c}_0(\NN)$ 
    by Mahler's theorem 
    (cf. \S\ref{sec:background}). 
To see that $\mathcal M \colon F \mapsto C$ is an isometry  
when $F$ is bounded 
it suffices to observe that the relation  
\eqref{eqn:defnfindiff}   
and its inverse 
\eqref{eqn:inverserelation} 
are both defined over $\ZZ$ and 
to apply the ultrametric inequality.  

Now suppose $m$ is a positive integer. Fix a 
non-negative integer $a$ and define the function $G_a$ by 
\begin{equation}
    G_{a}(n_0,\ldots,n_{m-1}) := F(n_0,\ldots,n_{m-1},a). 
\end{equation}
By induction, $G_a$ has a generalized Mahler series 
$$
\sum_{\ul{j} \in \NN^{m}} 
    D_{a}(\ul{j}) 
    \binom{\ul{x}}{\ul{j}} 
$$
    for uniquely determined $D_a \colon \NN^{m} \to E$. 
For fixed $\ul{j} \in \NN^m$, the function $a \mapsto D_{a}(\ul{j})$ 
also has a generalized Mahler series 
$\sum_{k \geq 0} d_{\ul j}(k) \binom{y}{k}$.  
We put these together to get 
\begin{equation}
    \label{eqn:mahlerseriesF}
\sum_{\ul{j} \in \NN^{m}} 
\sum_{k \geq 0}
    d_{\ul j}(k)
    \binom{\ul x}{\ul j}
    \binom{y}{k}
\end{equation}
which is therefore the generalized Mahler series of $F$, 
with Mahler coefficients given by 
$C(j_0,\ldots,j_m) := d_{(j_0,\ldots,j_{m-1})}(j_m)$. 
Uniqueness of $C$ follows from 
uniqueness of $d_{\ul j}$ and $D_{a}$. 


If $F$ is bounded then 
\begin{equation} 
\norm[F] = 
\sup_{\ul n \in \NN^{m+1}}
    |F(\ul n)|
=
\max_{\ul n \in \NN^{m+1}}
    |F(\ul n)|
=
\max_{a \in \NN} 
    \norm[G_{a}]
\end{equation} 
since $|E^\times| \subset \RR^{>0}$ is discrete. 
By induction, and discreteness of $|E^\times|$ again, 
\begin{equation} 
    \norm[F] = 
\max_{a \in \NN} 
    \norm[G_{a}] 
=
\max_{a \in \NN} 
    \norm[D_a]
=
\max_{\ul{j} \in\NN^m}
\max_{a \in \NN} 
    |D_{a}(\ul{j})|
=
\max_{\ul{j} \in\NN^m}
\max_{a \in \NN} 
    |d_{\ul{j}}(a)|
=
\norm[C].
\end{equation} 
This proves that $C$ is bounded if $F$ is bounded; 
conversely, if $C$ is bounded 
then $F$ is bounded by applying the ultrametric inequality to \eqref{eqn:finitesum}. 
When $F$ and $C$ are both bounded, 
we see that $\mathcal M \colon F \mapsto C$ 
is a self-isometry of $\Bdd(\NN^{m+1})$. 
The restriction of $\mathcal M$ 
to $\mathrm{C}(\NN^{m+1})$ is equal to $\mathcal M_0$ 
by uniqueness of Mahler coefficients. 
%
%
\end{proof} 


\subsection{Proof of supremum formula} 

Let $\seq\colon \NN \to E$ be a function and 
let $m$ be a non-negative integer. 
As before, let 
$\tau_{p,m}(n) = 
\max_{1 \leq i_1 < \cdots < i_m \leq n}
(\valn[p](i_1\cdots i_m))$
and 
$\findiff[n] := 
\sum_{k = 0}^n 
\binom{n}{k}(-1)^{n-k}\seq[k]. $

\begin{theorem}\label{prop:normdivdiff}
We have the equality  
\begin{equation}\label{eqn:lemmaeqnnorm}
    \norm[\dd{m}{\seq}]
=
\sup_{n \geq m}
|\findiff[n]|p^{\tau_{p,m}(n)}. 
\end{equation}
\end{theorem}

We will show that \eqref{eqn:lemmaeqnnorm} holds 
even if both sides are infinite. 

\begin{proof} { 

As $\dd{0}{\seq} = \seq$ and $\tau_{p,0}(n) = 0$, 
    the $m = 0$ case follows from Mahler's theorem 
    (cf. \S\ref{sec:background}). 
We suppose $m \geq 1$. 
Theorem~\ref{thm:boundedmahler} does not apply directly to $\dd{m}{\seq}$, 
as the domain of $\dd{m}{\seq}$ is $\Ndiag{m}$ and not all of $\NN^{m+1}$. 
We can resolve this with a change of variables as follows. 
Let $(a_0,\ldots,a_m) \in \NN^{m+1}$ and define 
\begin{align} 
    n_m &:= a_m, \\
    n_{m-1} &:= a_m + (a_{m-1}+1), \\
    \vdots & \\
    n_0 &:= a_m + (a_{m-1}+1) + \cdots + (a_0+1). 
\end{align} 
    This defines a bijection 
    $(a_0,\ldots,a_m) \leftrightarrow (n_0,\ldots,n_m) $ 
    from $\NN^{m+1}$ to the subset of $\Ndiag{m}$ 
    whose entries are strictly decreasing. 
Define 
\begin{align*}
    F \colon \NN^{m+1} &\to E \\
    (a_0,\ldots,a_{m}) &\mapsto 
    \dd{m}{\seq}(n_0,\ldots,n_m). 
\end{align*} 
Since $\dd{m}{\seq}$ is a symmetric function, 
    all of its values are realized on the subset 
    of $\Ndiag{m}$ whose entries are strictly decreasing,  
so $F$ takes the same values as $\dd{m}{\seq}$.  
In particular,  
    $\norm[F] = 
    \norm[\dd{m}{\seq}]$. 

We continue letting $\ul{a} \leftrightarrow \ul{n}$ 
    using the bijection above. 
To compute the Mahler coefficients of $F$ 
    we make use of the following formula of {Schikhof} 
    (cf. \cite{schikhof}, Theorem~54.1), 
\begin{multline}\label{eqn:phinewton}
    F(\ul{a}) = 
\sum_{\ul{j} \in\NN^{m+1}} 
    C({\ul{j}})
    \binom{\ul{a}}{\ul{j}} \quad\text{for all $\ul{a} \in \NN^{m+1}$}
    \\ 
    \text{where }
    C({\ul{j}}) = \frac
    {\findiff[j_0+\cdots+j_{m} + m]}
    {(j_{m-1}+1)(j_{m-1}+j_{m-2}+2)\cdots(j_{m-1}+\cdots+j_0+m)}. 
\end{multline}
By uniqueness of Mahler coefficients, 
    this is the generalized Mahler series for $F$ 
    (Theorem~\ref{thm:boundedmahler}).  

Set $i_k:=j_{m-1}+\cdots+j_{m-k}+k$. Then 
\begin{equation}\label{eqn:lemmapfeqn}
\norm[C] = 
\sup_{1 \leq i_1 < \cdots < i_m \leq \ell}
\left|
    \frac
        {\findiff[\ell]}
        {i_1i_2\cdots i_m}
\right|
=
\sup_{m \leq \ell}
    |\findiff[\ell]|p^{\tau_{p,m}(\ell)}
\end{equation}
where $\tau_{p,m}(\ell) 
:= 
\max_{1 \leq i_1 < \cdots < i_m \leq \ell}
(\valn[p](i_1)+\cdots + \valn[p](i_m)).
$
Theorem~\ref{thm:boundedmahler} asserts that 
$$\norm[F] = \norm[C].$$
Since $\norm[F] = \norm[\dd{m}{\seq}]$ this concludes the proof. 
%
%
}
\end{proof} 

\begin{corollary}\label{cor:ddmtolower}
Let $s \colon \NN \to \QQ$ be a function. 
If $\dd{m+1}{\seq}$ is $\ZZ$-valued then there is 
an explicit positive integer $C$ such that 
    $C\dd{k}{\seq}$ is $\ZZ$-valued 
    for all $k\in\{0,\ldots,m\}$. 
\end{corollary}

\begin{proof} 
The proposition implies that for any $k \in \{0,\ldots,m\}$, 
\begin{align}
\norm[\dd{k}{\seq}] = 
     \sup_{n \geq k}|c(n)|p^{\tau_{p,k}(n)}
    &=\sup\Big\{|c(k)|p^{\tau_{p,k}(k)},\sup_{n \geq k+1}|c(n)|p^{\tau_{p,k}(n)}\Big\}\\
    &\leq
    \sup\big\{ 
        |\findiff[k]||k!|^{-1},
    \norm[\dd{k+1}{\seq}]
    \big\}.
\end{align}
Thus 
\begin{equation} 
    \sup\big\{
        \norm[\dd{0}{\seq}],
        \norm[\dd{1}{\seq}],
        \ldots,
        \norm[\dd{m}{\seq}]
        \big\}
    \leq 
    \sup\big\{
        \big|\tfrac{c(0)}{0!}\big|,
        \big|\tfrac{c(1)}{1!}\big|,\ldots,
        \big|\tfrac{c(m)}{m!}\big|,1
        \big\}.
\end{equation} 
and we can take 
    \[ 
    C = 
    \prod_{p \text{ prime}}
    \sup\big\{
        \big|\tfrac{c(0)}{0!}\big|_p,
        \big|\tfrac{c(1)}{1!}\big|_p,\ldots,
        \big|\tfrac{c(m)}{m!}\big|_p,1
        \big\}.
    \] 
Note that only finitely many terms in the product differ from unity. 
\end{proof} 



\section{Asymptotic formulas for certain sums over primes} 
\label{sec:asymptotics} 

Let $\nfi$ denote a finite extension of $\QQ$. 
In the previous section we established 
a local estimate 
for the finite differences $\findiff$ of 
a function $\seq \colon \NN \to K$ with integral 
$m$th divided difference.  
Namely, for any nonarchimedean place $v$ of $K$ lying over $p_v$, 
Theorem~\ref{prop:normdivdiff} shows that 
\begin{equation*}
    \dd{m}{\seq} \text{ $\place$-integral} 
    \implies 
    |\findiff[n]|_\place \leq p_v^{-\tau_{p_\place,m}(n)}. 
\end{equation*}

Now suppose that $\dd{m}{\seq}$ is $v$-integral for 
\emph{all} nonarchimedean places $v$ of $K$. 
Then 
\begin{equation} 
    \prod_{v \text{ finite}} 
    |\findiff[n]|_\place \leq 
    \prod_{v \text{ finite}}
    p_v^{-\tau_{p_\place,m}(n)}
=
\exp \bigg\{
    -\sum_{v \text{ finite}}
        \tau_{p_\place,m}(n)\log p_v
    \bigg\}.
\end{equation} 
Our goal is to achieve a precise upper bound for 
$\prod_{v \text{ finite}} |\findiff[n]|_\place$, 
which will be used later to obtain a lower bound for 
$\prod_{v \text{ infinite}} |\findiff[n]|_\place$ 
by the product formula. 
Therefore we presently turn to the asymptotic growth of 
$\sum_{p \leq n}
    {\tau_{p,m}(n)}\log p$ 
in the limit of large $n$. 
The goal of this section is to prove that 
\begin{equation}\label{eqn:taugoalasymptotic} 
\sum_{p \leq n}
    {\tau_{p,m}(n)}\log p 
    =n \big(1+\tfrac12+\cdots+\tfrac1m\big) + o(n).
\end{equation} 

The standard bound for the Chebyshev function  
$\chebyshev(x) = x + O(\tfrac{x}{\log x})$ 
is not strong enough to establish the above asymptotic formula,  
so we employ 
the following useful estimate due to Rosser--Schoenfeld 
(cf. \cite{rosser-schoenfeld-1962}, (2.29)): 
\begin{equation}
    \label{eqn:chebyshevestimate}
    \chebyshev(x) = x + O(x \exp\{-\alpha (\log x)^{1/2}\}) 
\end{equation}
for some positive constant $\alpha$. 
In fact, this will lead to the improved error term 
$O(n \exp\{-\alpha (\log n)^{1/2} \} \log n )$ in place of $o(n)$ 
in \eqref{eqn:taugoalasymptotic}. 

\subsection{\texorpdfstring{A formula for $\tau_{p,m}$}{An explicit formula}} 

We begin by deriving an explicit formula for $\tau_{p,m}(n)$ 
which will be needed later.  
Let $\log_p$ denote the base-$p$ logarithm 
and let $[\,\cdot\,]$ denote the floor function. 
Let $a_p(n) := [n p^{-[\log_p n]}]$ 
    (the top nonzero $p$-adic digit in 
    the usual base-$p$ expansion of $n$). 
%
\begin{proposition}\label{prop:tauexplicitformula}
Let $m > 0$, $n \geq m$ be integers and let 
$p$ be a prime $\geq m$. 
    Then 
\begin{equation}\label{eqn:defntau}
\tau_{p,m}(n)=
\begin{cases*}
    m[\log_p n] & \text{if $a_p(n) \geq m$,} \\
    m[\log_p n] + a_p(n) - m & \text{if $a_p(n) < m$.}
\end{cases*}
\end{equation}
\end{proposition}

The formula can fail if $p < m$, e.g. 
$\tau_{p+1,p}(p^2) = p+1$ whereas \eqref{eqn:defntau} gives $p+2$. 

\begin{proof} 
    Write $t:= [\log_p n]$. 
    We will construct integers 
    $1\leq i_1 < \cdots < i_m\leq n$ 
    whose product realizes 
    the maximum $p$-adic valuation. 
    If $a_p(n) \geq m $ then set 
    $$i_k = (a_p(n)-m+k)p^t, \quad k \in \{1,\ldots,m\}.$$
    The $p$-adic valuation of $i_1\cdots i_m$ is $mt$ 
    and this is maximal since each integer has 
    the maximum possible $p$-adic valuation of 
    any positive integer $\leq n$. 
    If $a_p(n) < m$ then set 
    $$i_k = 
    \begin{cases}
        (a_p(n)-m+k)p^t &\text{if $k \in \{m-a_p(n)+1,\ldots,m\}$},\\
        p^t-(m-a_p(n)+1-k)p^{t-1} &\text{if $k \in \{1,\ldots,m-a_p(n)\}$}.
    \end{cases}
    $$
    As $p \geq m$ by hypothesis, 
    $m-a_p(n) \leq p-a_p(n) \leq p-1$ 
    and so $m-a_p(n)+1-k \in \{1,\ldots,p-1\}$ 
    for any $k \in \{1,\ldots,m-a_p(n)\}$. 
    Therefore the $p$-adic valuation of 
    $p^t-(m-a_p(n)+1-k)p^{t-1}$ is $t - 1$ 
    for any $k \in \{1,\ldots,m-a_p(n)\}$. 
    Putting these together obtains  
    \[\tau_{p,m}(n) = t a_p(n) + (t - 1)(m-a_p(n)) = mt + a_p(n) - m.\qedhere\]
\end{proof} 


\subsection{Asymptotic formulas for certain sums over primes} 

First we prove a simple lemma. 
Let $[\,\cdot\,]\colon \RR \to \ZZ$ denote the floor function. 

\begin{lemma}\label{lemma:fracpart} 
$$
\sum_{p \leq n}
    [\log_{p} n] \log p = n + 
    O(n \exp\{-\alpha (\log n)^{1/2}\}) 
$$
for some positive constant $\alpha$. 
\end{lemma} 

\begin{proof} 
    For any prime $p$ in the sum, 
    $r:=[\log_p n]$ must be positive. 
    Then 
    $$
    [\log_p n] = r 
    \iff
    \frac{\log n}{r+1} < \log p \leq \frac{\log n}{r}
    \iff
        n^{\frac{1}{r+1}} < p \leq n^{\frac{1}{r}}
    $$
    and
    \begin{align*}
    0 \leq \sum_{p \leq n}
     [\log_{p} n] \log p 
    &= 
    \sum_{r = 1}^\infty
    \sum_{n^{\frac{1}{r+1}} < p \leq n^{\frac{1}{r}} }
        r\log p\\
    &\leq 
    \sum_{\sqrt{n} < p \leq n}
        \log p\\
    &= 
    \chebyshev(n) - \chebyshev(\sqrt{n}). 
    \end{align*}
    The last term is 
    $n + O(n \exp\{-\alpha (\log n)^{1/2}\})$ 
    by \eqref{eqn:chebyshevestimate}. 
\end{proof} 

Let $\harm{m} = 1 + \tfrac{1}{2} + \cdots+ \tfrac{1}{m}$ be the $m$th harmonic number. 

\begin{theorem}\label{thm:harmonic} 
$$
\sum_{p \leq n}
    {\tau_{p,m}(n)}\log p 
= 
n \harm{m} + O(n \exp\{-\alpha (\log n)^{1/2} \} \log n ) 
$$
    for some positive constant $\alpha$. 
\end{theorem}

\begin{proof} 
This is clear if $m$ is zero 
so we suppose $m$ is positive. 
The explicit formula in Proposition~\ref{prop:tauexplicitformula} suggests that 
    we separate $\tau_{p,m}(n)$ into 
    a logarithmic term and a `remainder term'. 
The asymptotic contribution 
to $\sum_{p \leq n} \tau_{p,m}(n)\log p$ 
from the logarithmic term 
is established by Lemma~\ref{lemma:fracpart}: 
\begin{equation}\label{eqn:1}
    \sum_{p \leq n}{m[\log_p n]\log p} = 
    mn + O(n \exp\{-\alpha (\log n)^{1/2}\}).  
\end{equation}
The asymptotic contribution to the sum from the remainder 
 $m[\log_pn] - \tau_{p,m}(n)$ 
is significantly more difficult to establish. 
We will show that for some positive constant $\alpha$, 
\begin{equation}\label{eqn:2}
\sum_{p \leq n}
 \big( 
    m[\log_p n] - 
    \tau_{p,m}(n)
    \big) \log p
=
(m-\harm{m})n
+O(n \exp\{-\alpha (\log n)^{1/2} \} \log n ). 
\end{equation}

Taking the difference of  
\eqref{eqn:1} and \eqref{eqn:2} 
immediately proves the result so we now establish 
\eqref{eqn:2}.  
Let $a \in \{1,2,\ldots,m-1\}$ and $n \in \{m,m+1,\ldots\}$. 
For a prime $p$, let $a_p(n) = [n p^{-[\log_p n]}]$ 
    (the top nonzero $p$-adic digit in the base-$p$ expansion of $n$). 
We separate primes according to the value of $a_p(n)$. 
Define the set 
\begin{equation}
P_{a,n}:= 
    \{ p \text{ prime} : a = a_p(n),\,\,\, m \leq p \leq n
\}. 
\end{equation}
Consider the sum
$$
\gamma(n):=
\sum_{\substack{p\colon m \leq p \leq n,\\a_p(n) < m}}
    (m-a_p(n))\log p = 
    \sum_{a = 1}^{m-1}
\sum_{p \in P_{a,n}}
    (m-a)\log p . 
$$
Note that  
    $$\tau_{p,m}(n) \leq 
    \max_{1 \leq i_1 \leq \cdots \leq i_m \leq n}
        (\valn[p](i_1)+\cdots + \valn[p](i_m)) 
    = 
    m\max_{1 \leq i \leq n} \valn[p](i)
    =m[\log_p n],  
    $$
and so 
$\sum_{p < m}
 \big( 
    m[\log_p n] - 
    \tau_{p,m}(n)
    \big) \log p
    = O(\log n).$
Then by Proposition~\ref{prop:tauexplicitformula}, 
\begin{multline}\label{eqn:3} 
\sum_{p \leq n}
 \big( 
    m[\log_p n] - 
    \tau_{p,m}(n)
    \big) \log p\\
 = 
    O(\log n) +
    \sum_{\substack{m \leq p \leq n\\ a_p(n) < m}}
    (m-a_p(n))\log p 
    =\gamma(n) + O(\log n).
\end{multline} 

We will show that 
\begin{equation}\label{eqn:Gn3}
\gamma(n)
=
(m-\harm{m})n
+O(n \exp\{-\alpha (\log n)^{1/2} \} \log n ) 
\end{equation}
which will prove \eqref{eqn:2} in view of \eqref{eqn:3}. 

We claim that a prime $p$ is in $P_{a,n}$ if and only if 
    $p\geq m$ and $a = [np^{-t}]$ 
    for some $t \in \{1,2,\ldots,[\log_2 n]\}$.
Indeed, if $p\in P_{a,n}$ then $a = [np^{-t}]$ 
where $t = [\log_p n] \in \{1,2,\ldots,[\log_2 n]\}$. 
Conversely, if $a = [np^{-t}]$ 
for some $t \in \{1,2,\ldots,[\log_2 n]\}$, 
then $a \geq 1$ implies that $p^t \leq n$ and 
$p \geq m > a$ implies that $n < p^{t+1}$.  
Therefore $t = [\log_p n]$ and $a = a_p(n)$. 
This proves the claim. 
It follows that  
\begin{equation} 
\gamma(n) =  
    \sum_{a = 1}^{m-1}
        (m-a)
    \sum_{t = 1}^{[\log_2 n]} 
    \sum_{\substack{p\colon p \geq m,\\a=[np^{-t}]}}
        \log p.
\end{equation} 

Now observe that  
\begin{align}\label{eqn:equiv}
    a = [np^{-t}]
&\iff \nonumber
    a \leq np^{-t} < a+1\\ 
&\iff \nonumber
    an^{-1} \leq p^{-t} < (a+1)n^{-1}\\
&\iff  
    (n(a+1)^{-1})^{\tfrac{1}{t}}
     < p \leq
    (na^{-1})^{\tfrac{1}{t}}
\end{align}
and thus 
\begin{multline} 
\gamma(n) =  
    \sum_{a = 1}^{m-1}
        (m-a)
    \sum_{t = 1}^{[\log_2 n]} 
    \sum_{\substack{p\colon p \geq m,\\(n(a+1)^{-1})^{{1}/{t}} < p \leq (na^{-1})^{{1}/{t} }} }
        \log p
        \\
    =
    O(1)+
    \sum_{t = 1}^{[\log_2 n]} 
    \sum_{a = 1}^{m-1}
        (m-a)
        \big(
            \chebyshev(\{na^{-1}\}^{\tfrac{1}{t}}) - 
        \chebyshev(\{n(a+1)^{-1}\}^{\tfrac{1}{t}})
        \big)
\end{multline} 
where $\chebyshev(x) = \sum_{p \leq n} \log p$.
Let $\gamma_t(n)$ denote the inner sum of the last expression 
for $1 \leq t \leq [\log_2 n]$ so that $\gamma(n) = O(1) + \sum_{t = 1}^{[\log_2n]}\gamma_t(n)$. 
Then 
\newcommand{\tth}[2]{\chebyshev(\{\tfrac{n}{#1}\}^{1/#2})}
\begin{equation}
\label{eqn:Gtnchebyshev}
\gamma_t(n) 
    = m\tth{1}{t} -\tth{1}{t} - 
  \tth{2}{t} - \cdots - \tth{m}{t}. 
\end{equation}
By \eqref{eqn:chebyshevestimate} 
there is a positive constant $\alpha$ such that  
\begin{equation}\label{eqn:chebyshevestimate1t} 
\tth{a}{t} 
=
(\tfrac{n}{a})^{1/t} 
+ 
O(n^{1/t}\exp\{-\alpha (\log n)^{1/2} \}). 
\end{equation} 
Therefore 
$$\gamma_1(n) = (m-H_m)n+O(n\exp\{-\alpha (\log n)^{1/2} \}).$$
Since $n^{1/t} = O(n\exp\{-\alpha (\log n)^{1/2} \})$ for $t \geq 2$, 
we get that 
\begin{multline} 
    \gamma(n) 
    = \gamma_1(n) + O(n\exp\{-\alpha (\log n)^{1/2}\}\log n)\\
    =(m-H_m)n+ O(n\exp\{-\alpha (\log n)^{1/2}\}\log n).\qedhere
\end{multline} 
%
\end{proof} 

\begin{remark*} 
The proof shows that 
    any improved bound $\chebyshev(x)-x = O(f(x))$ 
    would lead to an improved bound of the form 
$nH_m-\sum_{p \leq n}
    {\tau_{p,m}(n)}\log p 
= 
    O(f(n)\log n )$. 
In particular, 
    if the Riemann hypothesis is true 
    then for any $\varep>0$, 
$$
\sum_{p \leq n}
    {\tau_{p,m}(n)}\log p 
= 
{n \harm{m} + O(n^{1/2+\varep} \log n )} .
$$
\end{remark*} 



\section{Proof of Theorem 
\ref{thm:introthm}}\label{sec:proofs} 

In this section we combine the local and global calculations 
from \S\ref{sec:divdiff} and \S\ref{sec:asymptotics} to prove 
the main theorem of the paper. 
It is easy to work with a general number field in place of $\mathbb Q$, 
and we choose to do so for the sake of generality. 
Let $\nfi$ be a finite extension of $\QQ$ 
of degree $\dgr{\nfi}$. 
Let $\primes{\nfi}$ denote the set of places of $\nfi$. 
For any place $\place \in \primes{\nfi}$ let 
$\localdgr{\place}$ denote the local degree at $\place$. 


\begin{lemma}\label{lemma:nullsequences} 
Let $\seq\colon \NN \to \nfi$ be a function and  
let $S \subset \primes{\nfi}$ be a finite set 
    containing the archimedean places.  
Suppose that  
\begin{enumerate}
    \item[\pnum{i}] $\dd{m}{\seq}$ is $S$-integer-valued\footnote{i.e., $\dd{m}{\seq}(\ul{n})$ is $\place$-integral 
        for all $\ul{n} \in \Ndiag{m}$ and $\place \not \in S$.}, and 
    \item[\pnum{ii}] for each $\place$ in $S$ 
        there is a positive number $\rho_\place$ such that 
    $|\findiff[n]|_\place \ll \rho_\place^{n}$. 
\end{enumerate}
    If 
\begin{equation}\label{eqn:prophypothesis}
\prod_{\place \in S} \rho_\place^{\localdgr{\place}} 
< 
e^{\dgr{\nfi}\big(1 + \tfrac12 + \cdots + \tfrac1m \big)}
\end{equation}
then $\seq[n]$ is a polynomial in $n$. 
\end{lemma} 

\begin{proof} 
    By Lemma~\ref{lemma:classical} 
    the conclusion is equivalent 
    to $\findiff$ being eventually zero, so  
    for the sake of contradiction suppose that 
    $\findiff$ has infinitely many nonzero terms. 

    Let $\place$ be a place of $\nfi$ not in $S$, 
    $p_\place$ the 
    rational prime $\place$ lies over, 
    $\embedding[\place]$ 
        a representative embedding for $\place$, 
    and $\localdgr{\place}$ 
    the local degree at $\place$. 
    By Theorem~\ref{prop:normdivdiff}, 
    $|\findiff[n]|_\place 
    = |\embedding[\place]\findiff[n]|_{p_\place} 
    \leq p_\place^{-\tau_{p_\place,m}(n)}$ for $n \geq m$, and so 
    \begin{equation}\label{eqn:cnthing}
    \prod_{\place\not \in S} 
        |\findiff[n]|_\place^{\localdgr{\place}} 
    \leq 
    \prod_{\place \not \in S}
        p_\place^{-\localdgr{\place}\tau_{p_\place,m}(n)}. 
    \end{equation}
    Note that both sides amount to 
    finite products   
    ($\tau_{p_\place,m}(n) = 0$ if $p_\place > n$). 

    Recall that the sum of local degrees at all places lying over 
    a given prime is the global degree, 
    and also that $\tau_{p,m}(n) = O(\log n)$.  
    Then 
    \begin{align*}
    \sum_{\place \not \in S}
        \localdgr{\place}\tau_{p_\place,m}(n) \log p_\place
    &=
    \sum_{p_\place \leq n}
        \localdgr{\place}\tau_{p_\place,m}(n) \log p_\place 
    -
    \sum_{\place \in S}
        \localdgr{\place}\tau_{p_\place,m}(n) \log p_\place \\
    &=
    \dgr{\nfi}
    \sum_{p \leq n}
        \tau_{p,m}(n) \log p
    +
    O(\log n). 
    \end{align*}
    Then with the help of Theorem~\ref{thm:harmonic} we 
    see that 
    $$
    \sum_{\place \not \in S}
        \localdgr{\place}\tau_{p_\place,m}(n) \log p_\place
    =
    \dgr{\nfi}n\harm{m}
    + o(n). 
    $$
    Putting this together with \eqref{eqn:cnthing} obtains 
    $$
    \limsup_{n \to \infty} 
    \prod_{\place \not \in S}
    |\findiff[n]|_\place^{\localdgr{\place}/n}
    \leq
    e^{-\dgr{\nfi}\harm{m}}. 
    $$

    Let $n_i$ be chosen so that $\findiff[n_i] \neq 0$ 
        for all non-negative integers $i$. 
    With the help of the product formula, 
    $\prod_{v \in M_K}|c(n_i)|_v^{d_v} = 1$, we obtain 
    \begin{equation*}
        \prod_{\place \in S}\rho_\place^{-\localdgr{\place}}
        \leq
    \liminf_{i\to\infty}
    \prod_{\place\in S} 
        |\findiff[n_i]|_\place^{-\localdgr{\place}/n_i} 
    =\limsup_{i \to \infty} 
    \prod_{\place \not \in S}
    |\findiff[n_i]|_\place^{\localdgr{\place}/n_i}
        \leq
        e^{-\dgr{\nfi}\harm{m}}. 
    \end{equation*}
    This contradicts \eqref{eqn:prophypothesis} 
    and concludes the proof. 
\end{proof} 

Now we will prove a generalization of 
Theorem~\ref{thm:introthm}. In addition to 
the integrality of higher divided differences 
we will simultaneously consider 
\emph{$p$-adic analytic interpolation}, i.e. 
the existence of a power series 
$f(x) \in \CC_p[\![x]\!] $ which converges for all 
$x \in 
\DD_{p}^{< R} := 
\{ x \in \CC_p : |x|_p < R \}$
    such that $R > 1$ and 
    $f(n)=\embedding[\place]\seq[n]$ for all $n\geq 0$, 
    where $\embedding[\place]$ is 
    a representative embedding for $\place$. 
    
%

\begin{theorem}\label{thm:general} 
Let $\seq\colon \NN \to \nfi$ be a function, 
let $S \subset \primes{\nfi}$ be a finite set 
containing the archimedean places $\primesarch{\nfi}$, and let   
    $T \subset \primes{\nfi}$ be a finite set disjoint from $S$. 
Suppose that  
\begin{enumerate}
    \item[\pnum{i}] $\dd{m}{\seq}$ is $(S\cup T)$-integer-valued,  
    \item[\pnum{ii}] for each $\place$ in $S$ there is a positive 
        number $\theta_\place$ such that 
    $|\seq[n]|_\place \ll \theta_\place^{n}$, and 
\item[\pnum{iii}] for each $\place$ in $T$ there is 
    a positive number $R_\place > 1$ 
    such that $\embedding[\place]{\seq}$ 
    extends to a $p$-adic analytic function 
    $\DD_{p_\place}^{< R_\place} \to \CC_{p_\place}$. 
\end{enumerate}
    If 
    \begin{equation}\label{eqn:MTIstatementinequality}
    \prod_{\place \in \primesarch{\nfi}} 
            (1+\theta_\place)^{\localdgr{\place}} 
    \prod_{\place \in S\backslash \primesarch{\nfi}} 
        \max\{1,\theta_\place\}^{\localdgr{\place}} 
    \prod_{\place \in T} 
        (p_\place^{\frac{1}{p_\place-1}} 
        R_\place)^{-\localdgr{\place}}
    < 
    e^{\dgr{\nfi}\big(1 + \tfrac12 + \cdots + \tfrac1m \big)}
        \end{equation}
        then $\seq[n]$ is a polynomial in $n$. 
\end{theorem} 

\begin{proof}[Proof of Theorem \ref{thm:general}] 

    By Lemma~\ref{lemma:nullsequences} 
    we see that $\seq$ is polynomial if 
    for each $\place \in S \cup T$ 
    there are positive constants 
    $D_\place$ and $\rho_\place$ 
    such that 
    \begin{equation}\label{eqn:neededineq}
    |\findiff[n]|_\place \leq D_\place\rho_\place^{n}
        \quad \text{for all $n \geq 0$}
    \end{equation}
     and 
    \begin{equation}\label{eqn:mainprodproofTMTI}
    \prod_{\place \in S \cup T}
    \rho_\place^{\localdgr{\place}} < 
        e^{\dgr{\nfi}\harm{m}}.
    \end{equation}

We can construct the $D_v$ and $\rho_v$ as follows.
First suppose that $\place$ is a nonarchimedean place in $S$. 
    As $|\seq[n]|_\place \ll \theta_\place^n$ by hypothesis, 
    there is a positive constant $D_\place$ such that   
$$
|\findiff[n]|_\place 
\leq
\max_{0 \leq k \leq n}
    |\seq[k]|_\place
\leq
\max_{0 \leq k \leq n}
D_\place \theta_\place^{k}
=
\begin{cases}
    D_\place       & \text{if $\theta_\place < 1$,    }\\ 
    D_\place\theta_\place^{n} & \text{if $\theta_\place \geq 1$.}
\end{cases}
$$
In either case we can take 
    $\rho_\place = \max\{ 1, \theta_\place \}$. 
    Now suppose $\place$ is a archimedean place in $S$. 
    We have that 
$$
|\findiff[n]|_\place
\leq
\sum_{0 \leq k \leq n}
\binom{n}{k}
|\seq[k]|_\place.
$$
Therefore for some positive constant $D_\place$ we have that 
    $
|\findiff[n]|_\place 
\leq
D_\place
(1 + \theta_\place)^n, 
$ 
so we can take $\rho_\place = 1+\theta_\place$. 
    Then \eqref{eqn:neededineq} is satisfied 
    for every $\place$ in $S$. 

Now we consider the places $\place$ in $T$. 
By hypothesis, for any place $\place$ in $T$ there exists 
a power series $f_{\place}(x) \in \CC_p[\![x]\!]$ 
converging on $\DD_{p_\place}^{< R_\place}$ such that  
$\embedding[\place]\seq[n] = f_{\place}(n)$ for all $n \geq 0$. 
Without loss of generality, we may assume that $f_\place$ 
converges on a disk of radius strictly larger 
than $R_\place$ for all $\place \in T$ since 
the inequality \eqref{eqn:MTIstatementinequality} remains 
valid even if $R_\place$ is replaced with a sufficiently close 
but smaller quantity; 
thus, there is an $\varep>0$ such that 
for all $\place$ in $T$, 
$f_\place$ converges on 
an open disk of radius $R_\place + \varep$.  
With the help of a theorem of Iwasawa 
(cf. \cite{iwasawa-1972}, Theorem~3), we see that 
\begin{equation}\label{eqn:inequalityfromiwasawa}
\lim_{n \to \infty}
    |\findiff[n]|_\place r^{-n} = 0
\end{equation}
for any real number $r$ such that 
\begin{equation}\label{eqn:inequality2fromiwasawa}
    p_\place^{\frac{-1}{p_\place-1}}(R_\place+\varep)^{-1} < r. 
\end{equation}
Then \eqref{eqn:inequalityfromiwasawa} implies that 
$|\findiff[n]|_\place^{1/n}>r$ for only finitely many $n$,  
and thus  
$$\limsup_{n \to \infty} |\findiff[n]|_\place^{1/n} \leq r.$$ 
This shows that for every $\place \in T$ 
there are positive constants $D_\place,\rho_\place$ satisfying 
\eqref{eqn:neededineq}, and 
additionally that $\rho_\place \leq r$. 
As $r$ was arbitrary subject 
to \eqref{eqn:inequality2fromiwasawa}, 
we may suppose 
$r \leq p_\place^{\frac{-1}{p_\place-1}}R_\place^{-1}$.  
Therefore 
the constants $D_\place, \rho_\place$ may be chosen to 
satisfy \eqref{eqn:neededineq} as well as the inequality 
\begin{equation}\label{eqn:rvcanalytic}
    \prod_{\place \in T}
    \rho_\place^{\localdgr{\place}}
\leq 
\prod_{\place \in T}
    (p_\place^{\frac{1}{p_\place-1}}R_\place)^{-\localdgr{\place}}. 
\end{equation}

Putting \eqref{eqn:rvcanalytic} together with 
the constructed $\rho_\place$ for $\place$ in $S$ shows that 
$$
    \prod_{\place \in S \cup T}
    \rho_\place^{\localdgr{\place}} 
    \leq 
    \prod_{\place \in \primesarch{\nfi}} 
            (1+\theta_\place)^{\localdgr{\place}} 
    \prod_{\place \in S\backslash \primesarch{\nfi}} 
        \max\{1,\theta_\place\}^{\localdgr{\place}} 
    \prod_{\place \in T} (p_\place^{\frac{1}{p_\place-1}} R_\place)^{-\localdgr{\place}}.
$$
This inequality shows that 
\eqref{eqn:MTIstatementinequality} 
implies 
\eqref{eqn:mainprodproofTMTI} and concludes the proof. 
\end{proof} 

\begin{corollary}\label{cor:introthmnumberfields}
Let $s \colon \NN \to \nfi$ be a function. 
Suppose that 
\begin{enumerate}
    \item[\pnum{i}] $\dd{m}{\seq}$ takes values in the ring of integers of $\nfi$, and 
    \item[\pnum{ii}]
        for each $\place \in \primesarch{\nfi}$ 
        there is a positive number $\theta_{\place}$ such that 
        $|\seq[n]|_v \ll \theta_{\place}^{n}$. 
\end{enumerate}
If 
$$
\prod_{\place \in \primesarch{\nfi}} 
        (1+\theta_\place)^{\localdgr{\place}} < 
        e^{ \dgr{\nfi} \big(
        1 + \tfrac{1}{2} + \cdots+ \tfrac{1}{m} 
        \big)} 
$$
then $\seq[n]$ is a polynomial in $n$. 
\end{corollary}

\begin{proof}[Proof of Corollary~\ref{cor:introthmnumberfields}]
Take $S = \primesarch{\nfi}$ and $T = \varnothing$ in Theorem~\ref{thm:general}.
\end{proof}

\begin{proof}[Proof of Theorem~\ref{thm:introthm}]
    Take $K = \QQ$ in Corollary~\ref{cor:introthmnumberfields}.
\end{proof}

\begin{remark*}
    It is well-known that a power series 
    $a_0 +a_1x + a_2x^2 + \cdots$ 
    is the expansion of a rational function if and only if 
    there exists a positive integer $N$ such that 
    \begin{equation}
        \label{eqn:dworkfindiff}
        \findiff[n]:=\det(a_{n+i+j})_{i,j = 0}^N
    \end{equation}
    is zero for all sufficiently large $n$. 
    Theorem~\ref{thm:general} may be applied to 
    the sequence $\seq$ whose finite differences are given 
    by \eqref{eqn:dworkfindiff} to obtain a generalization 
    of Theorems 2 and 3 of \cite{dwork-1960}. 
    We have not emphasized this application 
    however as hypotheses  
    $\pnum{i}$ and $\pnum{iii}$ of 
    Theorem~\ref{thm:general} 
    do not appear to be natural conditions on power series. 
\end{remark*}


\section{Integrality of divided differences}\label{sec:interpretations} 

In this section we give two interpretations of the integrality of higher divided differences. 
Our first interpretation generalizes the observation that 
the first divided difference $\dd{1}{\seq}$ of a function $\seq \colon \NN \to \QQ$ 
is integer-valued if and only if $\seq$ preserves all congruences. 
Roughly speaking, our first interpretation says that 
a function whose $m$th divided difference 
is integral is locally approximated to $m$th order by polynomials.  
Here `locally' refers to 
the topology on $\NN$ inherited 
from its inclusion into the ring of adeles. 
In this topology, the open neighborhoods are infinite arithmetic progressions 
and small neighborhoods are 
infinite arithmetic progressions with highly divisible periods. 

If $U \subset \NN$ is a subset, $g \colon U \to \QQ$ is any function, 
and $\varep$ is an integer, 
we write  
\begin{equation} 
    g(n) = O(\varep)
\end{equation} 
to mean $g(n)\in \varep\ZZ$ for all $n \in U$. 
For example, $g = O(1)$ if and only if $g$ is integer-valued. 


\begin{proposition}
Suppose $s \colon \NN \to \QQ$ is a function 
whose $m$th divided difference is integer-valued. 
Let $U = n_0 + \NN \varep \subset \NN$ be any infinite arithmetic progression. 
Then there is a polynomial $f(x) \in \QQ[x]$ of degree less than $m$ 
such that 
\begin{equation}
    \label{eqn:localpolynomial}
    s|_U(n) = f|_U(n) + O(\varep^{m}). 
\end{equation}
Furthermore, the denominators of the coefficients of $f(x)$ 
are divisors of 
    $$
    \prod_{p \text{ prime}}
    \sup\big\{
        \big|\tfrac{c(0)}{0!}\big|_p,
        \big|\tfrac{c(1)}{1!}\big|_p,\ldots,
        \big|\tfrac{c(m-1)}{(m-1)!}\big|_p,1
        \big\}$$ 
where $\findiff[n] := 
\sum_{k = 0}^n 
\binom{n}{k}(-1)^{n-k}\seq[k]$, 
    and 
\begin{equation} 
\gcd(f(n) : n \in U) = 
\gcd(\varep^kk!\cdot\dd{k}{\seq}(n_0,n_1,\ldots,n_k) : 0 \leq k < m).
\end{equation} 
\end{proposition}

\begin{proof} 
%
We make use of 
a classical interpolation formula due to Newton 
    (cf. e.g. \cite{milne-1951}, \S1). 
Let $\{n_0,n_1,\ldots\} \subset \QQ$ be an infinite subset 
and let $\seq \colon \{n_0,n_1,\ldots\} \to \QQ$ be a function. 
Then for all $n \in \{n_0,n_1,\ldots\}$, 
\begin{equation}
    \label{eqn:NIF}
    \seq[n] = \seq[n_0] + 
    \sum_{k = 1}^\infty
        \dd{k}{\seq}(n_0,n_1,\ldots,n_k)
        \prod_{j = 0}^{k-1}
            (n-n_j) \quad \text{(finite sum)}.
\end{equation}

Set $n_k := n_0 + k\varep$ for $k \in\NN$. 
From \eqref{eqn:NIF} we obtain that 
    \begin{equation}\label{eqn:NIFconsequence} 
    s(n) = f(n) + \sum_{k = m}^\infty
        \dd{k}{\seq}(n_0,n_1,\ldots,n_k)
        \prod_{j = 0}^{k-1}
            (n-n_j)
            \quad\text{for all $n \in U$}
\end{equation} 
    where $f(x):=s(n_0)+\sum_{k = 1}^{m-1}
    \dd{k}{\seq}(n_0,n_1,\ldots,n_k)
    \prod_{j = 0}^{k-1}
        (x-n_j)\in \QQ[x]$. 

We claim 
that for any integer $k\geq m$ 
and $a,a_0,a_1,\ldots\in U$,  
\begin{equation} 
    \dd{k}{\seq}(a_0,a_1,\ldots,a_k)
    \prod_{j = 0}^{k-1}
        (a-a_j)
\end{equation} 
is integral and divisible by $\varep^m$. 
We proceed by induction on $k$. 
The case $k = m$ is true by assumption. 
Suppose $k > m$. 
By the recursive formula for  
divided differences (cf. \cite{milne-1951}, \S1), 
\begin{equation} 
\dd{k}{\seq}(a_0,\ldots,a_k)
=
(a_1-a_0)^{-1}
    (\dd{k-1}{\seq}(\widehat{a_0},a_1,\ldots,a_{k})
    -\dd{k-1}{\seq}(a_0,\widehat{a_1},\ldots,a_{k}))
\end{equation} 
(the $\widehat{\cdot}$ indicates an omission). 
Then 
\begin{multline} 
\dd{k}{\seq}(a_0,a_1,\ldots,a_k)
\prod_{j = 0}^{k-1}
    (a-a_j)
=\\
    (\tfrac{a-a_{0}}{a_1-a_0})
    \Big\{
    \dd{k-1}{\seq}(\widehat{a_0},a_1,\ldots,a_{k})
    \prod_{j \in \{1,\ldots,k-1\}}
        (a-a_j)\Big\}\\
    -
    (\tfrac{a-a_{1}}{a_1-a_0})
    \Big\{
    \dd{k-1}{\seq}(a_0,\widehat{a_1},\ldots,a_{k})
    \prod_{j \in \{0,2,\ldots,k-1\}}
        (a-a_j)
        \Big\}.
\end{multline} 
By the inductive hypothesis, 
and the fact that $\dd{k-1}{\seq}$ is a symmetric function, 
each of the two terms in $\{\cdot\}$-brackets 
is integral and divisible by $\varep^m$. 
The remaining two factors are integers, and we have proven the claim. 
It follows from \eqref{eqn:NIFconsequence} 
that $s(n)-f(n)$ is integral and divisible by $\varep^m$ for any $n \in U$.     

The explicit bound on the denominators of $f$ follows from 
Corollary~\ref{cor:ddmtolower}. 
For the last claim, 
it is well-known that the binomial polynomials 
$\{\binom{y}{j}\}_{j \geq 0}$ form 
an orthonormal Banach basis of $C(\NN,\QQ_p)$ for all $p$. 
Observe that 
\begin{equation} 
    f(x)=\tilde{f}(y) =s(n_0)+\sum_{k = 1}^{m-1}
    \varep^k
    k! \cdot \dd{k}{\seq}(n_0,n_1,\ldots,n_k)
    \binom{y}{j}
    \quad
    \text{where $y:=\frac{x-n_0}{\varep}$.}
\end{equation} 
Therefore 
\begin{multline} 
    \gcd(f(n) : n \in U) = \gcd(\tilde{f}(n) : n \in \NN) \\
= \gcd(\varep^kk!\cdot\dd{k}{\seq}(n_0,n_1,\ldots,n_k) : 0 \leq k < m).\qedhere
\end{multline} 
\end{proof} 

%

Integrality of $\dd{1}{\seq}$ implies that  
for all primes $p$ and integers $m,n \in \NN$,  
$$
|\seq[m]-\seq[n]|_p \leq |m-n|_p.
$$
In other words, $\dd{1}{\seq}$ is $\ZZ$-valued 
if and only if $\seq$ is Lipschitz continuous 
with Lipschitz constant $1$ 
for every $p$-adic metric on $\NN$. 
Our second interpretation of the integrality of higher divided differences 
generalizes this observation. 

\begin{proposition}\label{prop:mahlercriterion}
    Let $\seq \colon \NN \to \QQ_p$ be a function and let 
    $m$ be a positive integer. 
    Suppose that $\norm[\dd{m}{\seq}]_p \leq 1$. 
    Then $\seq$ extends to an element $f$ 
    of $C^{m-1}(\ZZ_p,\QQ_p)$ and $f^{(m-1)}$ is 
    Lipschitz continuous 
    with Lipschitz constant $|(m-1)!|_p$. 
\end{proposition}
Here we mean \emph{strict} differentiability 
(cf. \cite{schikhof}, \S29).
\begin{proof} 
For $m = 1$ this was already explained 
    just before the proposition. 
Assume $m \geq 2$. 
By the recursive formula for 
divided differences (cf. \cite{milne-1951}, \S1), 
    \begin{equation}
        \label{eqn:recursivedefndivdiff}
|\dd{m-1}{\seq}(a_0,\ldots,a_{m-1}) - 
\dd{m-1}{\seq}(a_1,\ldots,a_m)| 
\leq
|a_0 - a_m|
    \end{equation}
for all $(a_0,\ldots,a_m) \in \Ndiag{m}$. 
Thus, 
\begin{equation}\label{eqn:manyineqdq}
    |\dd{m-1}{\seq}(a_i; a_{ij}) - 
    \dd{m-1}{\seq}(a_j; a_{ij})| 
\leq
|a_i - a_j| 
\end{equation}
where $0 \leq i < j \leq m$ and 
$a_{ij}:= 
(a_0,\ldots,\widehat{a_i},\ldots,
\widehat{a_j},\ldots, a_m) 
\in 
\Ndiag{m-2}$ 
(the $\widehat{\cdot}$ indicates an omission). 

We equip $\ZZ_p^m$ with the metric given by 
$$
d(x,y):= \max_{1 \leq i \leq m}|x_i-y_i|_p.
$$
We will show that 
$\dd{m-1}{\seq}$ 
is Lipschitz continuous for this metric on 
the dense subset 
$\Ndiag{m-1} \subset \ZZ_p^m$. 
By a limiting argument, 
it suffices to show the Lipschitz condition 
for $\dd{m-1}{\seq}$ on elements $x = (x_1,\ldots,x_m)$ 
and $y = (y_1,\ldots,y_m)$ in $\Ndiag{m-1}$ such that 
$x \cap y = \varnothing$. 
Define  
$z^0 := x$, $z^m := y$, and 
$$z^{i} := 
(x_1,x_2,\ldots,x_{m-i},y_{m-i+1},\ldots,y_m)\in\ZZ_p^m
\quad
    \text{for $0 < i < m$}.$$ 
Then  
$d(z^{i},z^{i+1}) 
= 
|x_{m-i}-y_{m-i}|_p$ for 
$0 \leq i < m$. 
Since $z^{i}$ and $z^{i+1}$ differ by only one coordinate 
and $x\cap y = \varnothing$, 
we may apply \eqref{eqn:manyineqdq}. We conclude 
that 
$$
|\dd{m-1}{\seq}(z^{i}) - \dd{m-1}{\seq}(z^{i+1})|_p 
\leq 
|x_{m-i}-y_{m-i}|_p = d(z^{i},z^{i+1})
$$ 
for $0 \leq i < m$. 
Using  
$\dd{m-1}{\seq}(z^0) - \dd{m-1}{\seq}(z^m)
= \sum_{i=0}^{m-1}
(\dd{m-1}{\seq}(z^{i}) - \dd{m-1}{\seq}(z^{i+1}))$ 
shows that 
\begin{align*}
|\dd{m-1}{\seq}(x) - \dd{m-1}{\seq}(y)|_p
&\leq 
\max_{0 \leq i < m}
|\dd{m-1}{\seq}(z^{i}) - \dd{m-1}{\seq}(z^{i+1})|_p \\
&\leq
\max_{0 \leq i < m}
|x_{m-i}-y_{m-i}|_p 
 \\
&=d(x,y). 
\end{align*}
 
We have shown that $\dd{m-1}{\seq}$ is 
Lipschitz continuous on a dense subset of $\ZZ_p^m$,  
so there is a unique Lipschitz continuous extension 
$\widetilde{\dd{m-1}{\seq}} \colon \ZZ_p^m \to \CC_p$ of 
$\dd{m-1}{\seq}$. 
Therefore $\seq$ extends to a $(m-1)$-times 
strictly differentiable function 
$f\colon \ZZ_p \to \QQ_p$ satisfying 
$$f^{(m-1)}(a) = 
(m-1)!\cdot \widetilde{\dd{m-1}{\seq}}(x_a)$$ where 
$x_a := (a,a,\ldots,a)$ (cf. \cite{schikhof}, \S29). 
By Lipschitz continuity of $\widetilde{\dd{m-1}{\seq}}$ 
we get that 
\[
|f^{(m-1)}(a) - f^{(m-1)}(b)|
\leq 
|(m-1)!|_p 
d(x_a,y_b)  
= 
|(m-1)!|_p 
|a-b|_p.\qedhere\]
\end{proof} 

\begin{corollary*}
Let $s \colon \NN \to \QQ$ be a function and suppose that 
    $\dd{m+1}{\seq}$ is $\ZZ$-valued for 
    a non-negative integer $m$. 
    Then for every prime $p$, the function $\seq$ extends 
    to an $m$-times strictly differentiable 
    function $f_p \colon \ZZ_p \to \QQ_p$ such that 
    $f_p^{(m)}$ is 
    Lipschitz continuous with Lipschitz constant $|m!|_p$. 
\end{corollary*}




\section*{Acknowledgements}

The author is grateful to Kartik Prasanna for 
a suggestion which simplified 
the proof of Theorem~\ref{thm:boundedmahler}. 
The author is also grateful to 
the anonymous referee for several valuable suggestions. 
The author would also like to thank Angus Chung and  
Trevor Hyde for helpful conversations. 


\bibliography{fd-local} 

\begin{thebibliography}{1}

\bibitem{amice-1964}
Y.~Amice.
\newblock Interpolation {$p$}-adique.
\newblock {\em Bull. Soc. Math. France}, 92:117--180, 1964.

\bibitem{bhargava-2009}
M.~Bhargava.
\newblock On {$P$}-orderings, rings of integer-valued polynomials, and
  ultrametric analysis.
\newblock {\em J. Amer. Math. Soc.}, 22(4):963--993, 2009.

\bibitem{dwork-1960}
B.~Dwork.
\newblock On the rationality of the zeta function of an algebraic variety.
\newblock {\em Amer. J. Math.}, 82:631--648, 1960.

\bibitem{hill-stra-1970}
D.~L. Hilliker and E.~G. Straus.
\newblock Some {$p$}-adic versions of {P}olya's theorem on integer valued
  analytic functions.
\newblock {\em Proc. Amer. Math. Soc.}, 26:395--400, 1970.

\bibitem{iwasawa-1972}
K.~Iwasawa.
\newblock {\em Lectures on {$p$}-adic {$L$}-functions}.
\newblock Princeton University Press, Princeton, N.J.; University of Tokyo
  Press, Tokyo, 1972.
\newblock Annals of Mathematics Studies, No. 74.

\bibitem{mahler}
K.~Mahler.
\newblock An interpolation series for continuous functions of a {$p$}-adic
  variable.
\newblock {\em J. Reine Angew. Math.}, 199:23--34, 1958.

\bibitem{milne-1951}
L.~M. Milne-Thomson.
\newblock {\em The {C}alculus of {F}inite {D}ifferences}.
\newblock Macmillan and Co., Ltd., London, 1951.

\bibitem{rosser-schoenfeld-1962}
J.~B. Rosser and L.~Schoenfeld.
\newblock Approximate formulas for some functions of prime numbers.
\newblock {\em Illinois J. Math.}, 6:64--94, 1962.

\bibitem{schikhof}
W.~H. Schikhof.
\newblock {\em Ultrametric calculus}, volume~4 of {\em Cambridge Studies in
  Advanced Mathematics}.
\newblock Cambridge University Press, Cambridge, 2006.

\end{thebibliography}
\bibliographystyle{abbrv}

\end{document}